\documentclass[12pt,twoside]{amsart}
\usepackage{amsmath}
\usepackage{amsthm}
\usepackage{amsfonts}
\usepackage{amssymb}
\usepackage{latexsym}
\usepackage{mathrsfs}
\usepackage{amsmath}
\usepackage{amsthm}
\usepackage{amsfonts}
\usepackage{amssymb}
\usepackage{latexsym}
\usepackage{geometry}
\usepackage{dsfont}
\usepackage[dvips]{graphicx}
\usepackage[colorlinks=true,linkcolor=red,citecolor=blue]{hyperref}
\usepackage{color}
\usepackage[all]{xy}

\date{}
\allowdisplaybreaks[4] \footskip=15pt
\renewcommand{\uppercasenonmath}[1]{}

\numberwithin{equation}{section} \theoremstyle{plain}
\usepackage{graphicx,amssymb}
\usepackage[all]{xy}
\usepackage{amsmath}

\allowdisplaybreaks
\usepackage{amsthm}
\usepackage{color}

\theoremstyle{plain}

\theoremstyle{plain}
\newtheorem{theorem}{Theorem}[section]
\newtheorem{proposition}[theorem]{Proposition}
\newtheorem{lemma}[theorem]{Lemma}

\newtheorem*{open question}{Open Question}

\theoremstyle{definition}
\newtheorem*{acknowledgement}{Acknowledgement}

\theoremstyle{remark}
\newtheorem{remark}[theorem]{Remark}

\newcommand{\Tor}{\mbox{\rm Tor}}

\newcommand{\Q}{\mathcal{Q}}

\newcommand{\Id}{\mathrm{Id}}

\def\pd{{\rm pd}}

\def\Reg{{\rm Reg}}

\def\tor{{\rm tor}}
\def\Hom{{\rm Hom}}
\def\Ext{{\rm Ext}}
\def\Tor{{\rm Tor}}

\def\Ker{{\rm Ker}}

\def\Reg{{\rm Reg}}

\def\T{{\rm T}}
\def\E{{\rm E}}

\newcommand{\p}{\frak{p}}

\def\Min{{\rm Min}}

\begin{document}
\begin{center}
{\large  \bf Characterizing Baer Modules Over $\tau_q$-Semisimple Rings: An Extension of Baer's Problem}

\vspace{0.5cm}
Xiaolei Zhang\\

E-mail: zxlrghj@163.com\\

\vspace{0.5cm}
Hwankoo Kim\footnote{Corresponding author}\\

E-mail:  hkkim@hoseo.edu\\
\end{center}

%\begin{figure}[b]
%\rule[-2.5truemm]{5cm}{0.1truemm}\\[2mm]
%{\small }
%\end{figure}

%\begin{figure}[b]
%\rule[-2.5truemm]{5cm}{0.1truemm}\\[2mm]
%{\small }
%\end{figure}
\bigskip
\centerline { \bf  Abstract}
\bigskip
\leftskip10truemm \rightskip10truemm \noindent
In this note, we investigate the Baer splitting problem over commutative rings. In particular, we show that if a commutative ring $R$ is $\tau_q$-semisimple, then every Baer $R$-module is projective.
\\
\vbox to 0.3cm{}

\leftskip0truemm \rightskip0truemm
\bigskip
\section{Introduction}

All rings in this note are assumed to be commutative with identity, and all modules are unitary.

In early 1936, Baer\cite{B36} posed the following classical question: How to characterize the class of all abelian groups $G$ such that every extension of $G$ with a torsion group splits, i.e., to characterize abelian groups $G$ satisfying $\Ext_{\mathbb{Z}}^1(G, T)=0$ for any torsion groups $T$ in modern language \cite{HBH08}. Baer\cite{B36} himself proved that every countably generated group $G$ with this property must be free \cite{HBH08}. In 1962, Kaplansky\cite{K62} extended this question to commutative domains and called the modules satisfying this property Baer modules. Let $D$ be a commutative domain. A $D$-module $P$ is said to be a Baer module if $\Ext_D^1(P,T)=0$ for any torsion $D$-module $T$. Kaplansky\cite{K62} conjectured that every Baer module over a commutative domain is projective. In 2008, Angeleri H\"ugel, Bazzoni, and Herbera proved this conjecture by translating the vanishing of Ext into a Mittag-Leffler condition on inverse systems \cite{HBH08}.

A natural question is how to extend this conjecture to commutative rings with zero-divisors. In fact, Golan \cite[Chapter 11]{G86}, in 1986, introduced the notion of $\sigma$-projective modules over general rings, where $\sigma$ is a given torsion theory. An $R$-module $P$ is said to be $\sigma$-projective if $\Ext_R^1(P,T)=0$ for any $\sigma$-torsion module $T$. So it is interesting to characterize rings over which all $\sigma$-projective modules are projective.

In this note, we consider a classical torsion theory over commutative rings: An $R$-module $T$ is said to be a torsion module if, for any $t\in T$, there exists a non-zero-divisor $r\in R$ such that $rt=0$. Thus, torsion modules over commutative rings are natural generalizations of those over commutative domains. We denote this torsion theory by $reg$ and call the notion of $reg$-projective module in the sense of Golan to be a Baer module again to continue the classical situation. Recently, Zhang \cite{z-q-proj-sim} introduced the notion of a $\tau_q$-semisimple ring $R$, which is equivalent to its total ring $\T(R)$ being a semisimple ring (see Section 2 for more details). In this note, we prove that if $R$ is a $\tau_q$-semisimple ring, then every Baer $R$-module is projective.

Any undefined notation or terminology can be found in \cite{WK24}.

\section{Preliminaries on $\tau_q$-semisimple rings}
Recall some notions about (strongly) Lucas modules, $\tau_q$-projective modules, and $\tau_q$-semisimple rings (see \cite{z-q-proj-sim, wzcc20,  ZDC20, ZQ23}). An ideal $I$ of a ring $R$ is said to be \emph{dense} if $(0:_R I) := \{r \in R \mid Ir = 0\} = 0$; \emph{regular} if $I$ contains a non-zero-divisor; and \emph{semi-regular} if there exists a finitely generated dense sub-ideal of $I$. Denote by $\Reg(R)$ the set of all regular elements of $R$. The set of all finitely generated semi-regular ideals of $R$ is denoted by $\Q$. Let $M$ be an $R$-module. Denote by
$$\tor_{\Q}(M):=\{m\in M \mid Im=0 ~\mbox{for some}~ I\in \Q \}.$$
An $R$-module $M$ is said to be \emph{$\Q$-torsion} (resp., \emph{$\Q$-torsion-free}) if $\tor_{\Q}(M) = M$ (resp., $\tor_{\Q}(M) = 0$). A $\Q$-torsion-free module $M$ is called a \emph{(strongly) Lucas module} if $\Ext_R^1(R/I, M) = 0$ for every $I \in \Q$ (and every $n \geq 1$). An $R$-module $M$ is said to be \emph{$\tau_q$-projective} if $\Ext_R^1(M, N) = 0$ for every strongly Lucas module $N$. A ring $R$ is said to be \emph{$\tau_q$-semisimple} if every $R$-module is $\tau_q$-projective. Trivially, integral domains and semisimple rings are $\tau_q$-semisimple rings.

Recall from \cite{fl11} that an $R$-module $M$ is said to be \emph{reg-injective} if $\Ext_R^1(R/I, M) = 0$ for every regular ideal $I$ of $R$; and $M$ is said to be \emph{semireg-injective} if $\Ext_R^1(R/I, M) = 0$ for every semi-regular ideal $I$ of $R$. Trivially injective modules are semireg-injective, and semireg-injective modules are reg-injective. Then the following result shows when these three notions coincide.

\begin{lemma}\label{0-d}\cite[Theorem 4.3]{z-q-proj-sim}
The following statements are equivalent for a ring $R$.
\begin{enumerate}
\item $R$ is a $\tau_q$-semisimple ring.
\item $\T(R[x])$ is a semisimple ring.
\item $R[x]$ is a reduced ring with finite minimal prime ideals.
\item $\T(R)$ is a semisimple ring.
\item $R$ is a reduced ring with finite minimal prime ideals.
\item Every reg-injective module is injective.
\item Every semireg-injective module is injective.
\item Every Lucas module is injective.
\item Every strongly Lucas module is injective.
\end{enumerate}
\end{lemma}

From Lemma \ref{0-d}, we see that reduced Noetherian rings are $\tau_q$-semisimple. Moreover, the polynomial, formal power series, and amalgamation transfer of $\tau_q$-semisimplicity can be found in \cite[Section 4]{z-q-proj-sim}.

It is well known that the total ring $\T(R)$ of quotients of an integral domain is both flat and injective. The following result shows that this is also true for $\tau_q$-semisimple rings.

\begin{lemma}\label{fiTR}
Let $R$ be a $\tau_q$-semisimple ring. Then $\T(R)$ is both flat and injective as an $R$-module.
\end{lemma}

\begin{proof}
Obviously, $\T(R)$ is a flat $R$-module. To show that $\T(R)$ is injective, it suffices to show by Theorem \ref{0-d} that $\T(R)$ is reg-injective since $R$ is a $\tau_q$-semisimple ring. Let $I$ be a regular ideal of $R$ with a non-zero-divisor $r$. Let $\xi: 0 \rightarrow \T(R) \rightarrow X \rightarrow R/I \rightarrow 0$ be an extension in $\Ext_R^1(R/I, \T(R))$. Then multiplying $r$ to the left of $\xi$ is an isomorphism, but to the right is zero. Hence $\Ext_R^1(R/I, \T(R)) = 0$ by \cite[Chapter I, Section 5]{FS01}. So $\T(R)$ is reg-injective.
\end{proof}

\begin{lemma}\label{fgnot}
Let $R$ be a $\tau_q$-semisimple ring that does not contain a field as a direct summand. Then every nonzero direct summand of $\T(R)$ is not finitely generated as an $R$-module.
\end{lemma}

\begin{proof}
Since $R$ is a $\tau_q$-semisimple ring, we can assume $\Min(R) = \{\p_i \mid i = 1, \dots, n\}$. It follows by \cite[Proposition 1.1, Proposition 1.5]{M83} that $\T(R) \cong R_{\p_1} \times \cdots \times R_{\p_n}$, where each $R_{\p_i}$ is a field. On the contrary, suppose there is some $i$ such that $R_{\p_i}$ is a finitely generated $R$-module, say generated by ${\frac{1}{r_1}, \dots, \frac{1}{r_n}}$, where each $r_j \not\in \p_i$. Since $R_{\p_i}$ is a flat $R$-module, it is also projective by \cite[Lemma 3.1]{PR04}. Set $r := r_1 \cdots r_n \not\in \p_i$. Then there is an epimorphism $R \twoheadrightarrow R_{\p_i}$, which maps $s$ to $\frac{s}{r}$ for any $s \in R$. Note that the epimorphism is splitting, since $R_{\p_i}$ is projective. Thus, $R$ contains a field as a direct summand, which is a contradiction.
\end{proof}

Recall that an $R$-module $M$ is said to be torsion-free if $rm=0$ with $r \in \Reg(R)$ and $m \in M$ implies that $m=0$.

\begin{lemma}\label{fg0}(cf. \cite[Lemma 3.2]{HBH08})
Let $R$ be a $\tau_q$-semisimple ring which does not contain a field as a direct summand. Suppose $M$ is a finitely generated $R$-module. Then $\bigcap\limits_{r \in \Reg(R)} rM = 0$.
\end{lemma}

\begin{proof}
First, we assume that $M$ is finitely generated and torsion-free. Set $d(M):=\bigcap\limits_{r \in \Reg(R)} rM$. Then $d(M)$ is torsion-free and divisible, thus a $\T(R)$-module. On the contrary, suppose $d(M) \neq 0$. We can also assume $\Min(R) = \{\p_i \mid i=1, \dots, n\}$. Then $\T(R) \cong Q_1 \times \cdots \times Q_n$, where each $Q_i \cong R_{\p_i}$ is a field. So there exists a field $Q_i$, which is a direct summand of $d(M)$. By Lemma \ref{fiTR}, $Q_i$ is an injective submodule, and hence a direct summand of $M$. However, by Lemma \ref{fgnot}, $Q_i$ is not finitely generated, which is a contradiction.

Now suppose that $M$ is not torsion-free. Consider the following exact sequence:
$$0\rightarrow t(M)\rightarrow M\rightarrow M/t(M)\rightarrow0,$$
where $t(M)$ is the torsion submodule of $M$. Let $m \in \bigcap\limits_{r \in \Reg(R)} rM$. Then
$$m+t(M)\in \bigcap\limits_{r\in\Reg(R)}(rM+t(M))/t(M)=\bigcap\limits_{r\in\Reg(R)}M/t(M)=0$$
by the proof of the first part. So $m \in t(M)$. It is easy to verify $rM \cap t(M) = rt(M)$ for every $r \in \Reg(R)$. So we have $m \in \bigcap\limits_{r \in \Reg(R)} rt(M)$. Since $M$ is finitely generated, $t(M)$ can be annihilated by some $r \in \Reg(R)$. Hence $\bigcap\limits_{r \in \Reg(R)} rt(M) = 0$, and so $m = 0$. Thus, $\bigcap\limits_{r \in \Reg(R)} rM = 0$.
\end{proof}

\begin{proposition}\label{fg1}(cf. \cite[Lemma 3.3]{HBH08})
Let $R$ be a $\tau_q$-semisimple ring which does not contain a field as a direct summand. Then the map
$$\mu:R\rightarrow\prod\limits_{r\in\Reg(R)} R/rR$$
defined by $\mu(x) = (x + rR)$ for any $x \in R$ is a pure embedding.
\end{proposition}

\begin{proof}
Note that $\Ker(\mu) = \bigcap\limits_{r \in \Reg(R)} rR$. It follows by Lemma \ref{fg0} that $\mu$ is an embedding. To show that $\mu$ is pure, it suffices to show that
$$\mu\otimes_R\Id_M:R\otimes_RM\rightarrow \prod\limits_{r\in\Reg(R)} R/rR\otimes_RM$$
is a monomorphism for any finitely presented $R$-module $M$. Since $M$ is finitely presented, we have
$$\prod\limits_{r\in\Reg(R)} R/rR\otimes_RM\cong \prod\limits_{r\in\Reg(R)} M/rM.$$
Hence $\Ker(\mu \otimes_R \Id_M) \cong \bigcap\limits_{r \in \Reg(R)} rM$, which is equal to $0$ by Lemma \ref{fg0}. So $\mu$ is a pure embedding.
\end{proof}

\section{Baer modules over $\tau_q$-semisimple rings}

Recall that an $R$-module $T$ is said to be a torsion module if, for any $t \in T$, there exists a non-zero-divisor $r \in R$ such that $rt = 0$.

%\begin{definition}
%An $R$-module $M$ is said to be a Baer module if $\Ext_R^1(M, T) = 0$ for every torsion $R$-module $T$.
%\end{definition}

Trivially, projective modules are Baer modules. Moreover, it was proved in \cite{HBH08} that every Baer module is projective over integral domains. The main motivation of this paper is to extend this result to $\tau_q$-semisimple rings. First, we investigate some basic properties of Baer modules over $\tau_q$-semisimple rings.

\begin{proposition}\label{baerflat}(cf. \cite[Lemma 1]{EF88})
Let $R$ be a $\tau_q$-semisimple ring. Then every Baer $R$-module is flat.
\end{proposition}

\begin{proof}
Let $M$ be a Baer $R$-module and $I$ be a finitely generated ideal of $R$. We will show $\Tor_1^R(R/I, M) = 0$ by dividing it into two cases:
	
\textbf{Case 1}:  Suppose $I$ is a regular ideal of $R$. Then $(R/I)^+ := \Hom_{\mathbb{Z}}(R/I, \mathbb{Q}/\mathbb{Z})$ is certainly a torsion $R$-module. Hence $(\Tor_1^R(R/I, M))^+ \cong \Ext_R^1(M, (R/I)^+) = 0$, and so $\Tor_1^R(R/I, M) = 0$.

\textbf{Case 2}: Suppose $I$ is not a regular ideal of $R$. Then by the proof of \cite[Theorem 4.3]{z-q-proj-sim}, there exists $s \in R$ such that $I \oplus Rs$ is a regular ideal of $R$. Thus $\Tor_1^R(R/(I \oplus Rs), M) = 0$ by the proof of Case 1. Hence the natural map $(I \oplus Rs) \otimes_R M \rightarrow M$ is a monomorphism, and so is the natural map $I \otimes_R M \rightarrow M$. Thus $\Tor_1^R(R/I, M) = 0$.

Consequently, $M$ is a flat $R$-module.
\end{proof}

\begin{proposition}\label{baerpd1}(cf. \cite[Lemma 2]{EF88})
Let $R$ be a $\tau_q$-semisimple ring. Then every Baer $R$-module has projective dimension $\leq 1$.
\end{proposition}

\begin{proof}
Let $M$ be a Baer $R$-module and $N$ be an $R$-module. Then there exists a reg-injective envelope $E$ of $N$ such that $E/N$ is torsion. Indeed, by setting $E := \{e \in \E(M) \mid re \in M \text{ for some } r \in \Reg(R)\}$, where $\E(M)$ is the injective envelope of $M$, it is a routine way to show that $E$ is a reg-injective module such that $E/N$ is torsion. Since $R$ is a $\tau_q$-semisimple ring, the reg-injective $R$-module $E$ is injective by Lemma \ref{0-d}. Considering the exact sequence
$$0=\Ext_R^1(M,E/N)\rightarrow\Ext_R^2(M,N)\rightarrow\Ext_R^2(M,E)=0,$$
we have $\Ext_R^2(M, N) = 0$. Hence $\pd_R M \leq 1$.
\end{proof}

Let $M$ be an $R$-module. An $R$-submodule $N$ of $M$ with $\pd_R(N) \leq 1$ is said to be tight if $\pd_R(M/N) \leq 1$.

\begin{lemma}(cf. \cite[Lemma 3]{EF88})
Every tight submodule of a Baer module is a Baer module.
\end{lemma}

\begin{proof}
Let $M$ be a Baer module and $N$ be a tight submodule of $M$. Let $T$ be a torsion module. Consider the following exact sequence $$0=\Ext_R^1(M,T)\rightarrow\Ext_R^1(N,T)\rightarrow\Ext_R^2(M/N,T)=0.$$
We have $\Ext_R^1(N, T) = 0$, and hence $N$ is a Baer module.
\end{proof}

\begin{lemma}\label{1bapd1}(cf. \cite[Lemma 4]{EF88})
An $R$-module $B$ of projective dimension $\leq 1$ is a Baer module exactly if $\Ext_R^1(B, T) = 0$ for all direct sums $T$ of $\bigoplus\limits_{r \in \Reg(R)} (R/Rr)$.
\end{lemma}

\begin{proof}
Let $N$ be a torsion module. Set $T_{g} := \bigoplus\limits_{r \in \Reg(R)} (R/Rr)$. Then there is an exact sequence
$$0\rightarrow K\rightarrow T_g^{(\kappa)}\rightarrow N\rightarrow0.$$
Considering the exact sequence
$$0=\Ext_R^1(B,T^{(\kappa)})\rightarrow\Ext_R^1(B,N)\rightarrow\Ext_R^2(B,K)=0,$$
we have $\Ext_R^1(B, N) = 0$. Hence $B$ is a Baer module.
\end{proof}

\begin{lemma}\label{gencard}
Let $R$ be a $\tau_q$-semisimple ring. Let $M$ be an infinitely generated projective $R$-module and $N$ be an infinitely generated projective submodule of $M$. Then the minimum cardinality of the generators of $N$ is less than or equal to that of $M$.
\end{lemma}

\begin{proof}
Assume that $\Min(R) = \{\p_i \mid i=1, \dots, n\}$. Let $\T(R)$ be the total quotient ring of $R$. Then $\T(R) \cong R_{\p_1} \times \cdots \times R_{\p_n}$, where each $R_{\p_i}$ is a field. Note that the minimum cardinality of generators of an infinitely generated projective $R$-module $M$ is equal to the maximal dimension of $\T(R) \otimes_R M$ over all its components $R_{\p_i}$'s. Since $\T(R) \otimes_R N$ is a $\T(R)$-submodule of $\T(R) \otimes_R M$, the result is easily obtained.
\end{proof}

\begin{proposition}\label{2bapd1}(cf. \cite[Lemma 5]{EFS90})
Let $R$ be a $\tau_q$-semisimple ring. If an $R$-module $B$ of projective dimension $\leq 1$ can be generated by $\kappa$ (where $\kappa$ is an infinite cardinal) elements, then for $T_0 = \bigoplus\limits_{\kappa} \left(\bigoplus\limits_{r \in \Reg(R)} R/Rr \right)$, $\Ext_R^1(B, T_0) = 0$ implies that $B$ is a Baer module.
\end{proposition}

\begin{proof}
Let
$$0\rightarrow H\xrightarrow{\alpha} F\rightarrow B\rightarrow 0$$
be a projective resolution of $B$, where $F$ is free with $\kappa$ generators. Since $R$ is a $\tau_q$-semisimple ring, $H$ is also projective with at most $\kappa$ generators by Lemma \ref{gencard}. We can assume that $H$ is a free module. Consider the induced sequence $$\Hom_R(F,T)\xrightarrow{\alpha^{\ast}}\Hom_R(H,T)\rightarrow\Ext_R^1(B,T)\rightarrow0,$$
where $T$ is a direct sum of $\bigoplus\limits_{r \in \Reg(R)} R/Rr$. Since every morphism $\eta$ in $\Hom_R(H, T)$ lands in a $\kappa$-generated summand of $T$. Such a summand is isomorphic to a summand of $T_0$. So $\eta$ must be induced by a homomorphism in $\Hom_R(F, T)$. Hence $\alpha^{\ast}$ is an epimorphism and so $\Ext_R^1(B, T) = 0$. Hence $B$ is a Baer module by Lemma \ref{1bapd1}.
\end{proof}

\section{main results}

Let $\gamma$ be a limit ordinal. A subset $C$ of $\gamma$ is called a club (in $\gamma$) if (1) $C$ is closed in $\gamma$, i.e., for all $Y \subseteq C$, if $\sup Y \in \gamma$, then $\sup Y \in C$; and (2) $C$ is unbounded in $\gamma$, i.e., $\sup C = \gamma$. A set $X$ is said to be stationary in $\gamma$ if for all clubs $C$, $X \cap C \ne \emptyset$.

From its proof, the following result holds for all rings.

\begin{lemma}\cite[Lemma 9]{EFS90}\label{filt}
Let $M$ be a $\kappa$-generated $R$-module (where $\kappa$ is a regular uncountable cardinal) and suppose
\begin{equation}\label{(1)}
0 = M_0<M_1<\cdots <M_{\nu}<\cdots \ (\nu < \kappa)
\end{equation}
is a well-ordered continuous ascending chain of submodules such that

$(i)$ $M_\alpha$ is less than $\kappa$-generated for each $\alpha < \kappa$;

$(ii)$ $\pd_R (M_{\alpha+1}/M_{\alpha}) \leq 1$ for each $\alpha < \kappa$;

$(iii)$ $\bigcup\limits_{\alpha < \kappa} M_{\alpha} = M.$

For each $\alpha < \kappa$, let $G_{\alpha}$ be an $R$-module, and set $H_{\beta} := \bigoplus\limits_{\alpha < \beta} G_{\alpha}$ for $\beta < \kappa$, and $H := \bigoplus\limits_{\alpha < \kappa} G_{\alpha}$. If the set
	$$E :=\{\alpha<\kappa \mid \exists\beta>\alpha\ \mbox{with}\  \Ext_R^1(M_{\beta}/M_{\alpha},H_{\beta}/H_{\alpha})\not=0\}$$
is stationary in $\kappa$, then $\Ext_R^1(M, H) \ne 0$.
\end{lemma}

Recall from \cite{F23} that a ring $R$ is called subperfect if its total quotient ring $\T(R)$ is a perfect ring. Certainly, every $\tau_q$-semisimple ring is subperfect. It was proved in \cite[Chapter VI, Proposition 5.1]{FS01} that every module of projective dimension $\leq 1$ over an integral domain has a tight system. Recently, Fuchs showed that it also holds over subperfect rings.

\begin{lemma}\label{tight-sys}
\cite[Theorem 3.1]{F23} Let $R$ be a subperfect ring and $B$ be an $R$-module of projective dimension $\leq 1$. Then $B$ has a tight system, i.e., a collection $\mathscr{T}$ of submodules of $B$ such that
\begin{enumerate}
\item $0, B \in \mathscr{T}$;
\item $\mathscr{T}$ is closed under unions of chains;
\item if $B_i < B_j$ in $\mathscr{T}$, then $\pd_R (B_j/B_i) \leq 1$;
\item if $B_i \in \mathscr{T}$ and $C$ is a countable subset of $B$, then there is a $B_j \in \mathscr{T}$ such that $\langle B_i, C \rangle \subseteq B_j$ and $B_j/B_i$ is countably generated.
\end{enumerate}
\end{lemma}

The authors show that the following result holds for integral domains (see \cite[Theorem A]{EFS90}). We can prove it under $\tau_q$-semisimple rings along their proof.

\begin{theorem}\label{Baer-filt}
Let $R$ be a $\tau_q$-semisimple ring. Then an $R$-module $B$ is a Baer module if and only if there exists a well-ordered continuous ascending chain of submodules
$$0 = B_0<B_1<\cdots <B_{\nu}<\cdots <B_{\tau} = B,\ (\nu < \tau)$$
for some ordinal $\tau$ such that, for every $\nu < \tau$, $B_{\nu+1}/B_{\nu}$ is a countably generated Baer module.
\end{theorem}

\begin{proof}
Sufficiency holds by the Auslander-Eklof lemma (see \cite[Lemma 6.2]{gt}).

For necessity, we apply induction on $\kappa$ to the cardinality of a set of generators of the Baer module $B$. If $\kappa \leq \aleph_0$, then there is nothing to prove. Now let $\kappa \geq \aleph_1$, and assume that the claim holds for Baer modules with fewer than $\kappa$ generators.

\textbf{Case 1}: $\kappa$ is a regular cardinal. Since $B$ is a Baer module, $\pd_R B \leq 1$ by Proposition \ref{baerpd1}. Hence $B$ has a tight system by Lemma \ref{tight-sys}. Set $B = M$, and there exists a sequence \ref{(1)} which satisfies the conditions (i)-(iii) in Lemma \ref{filt}. To apply Lemma \ref{filt}, choose $G_{\alpha}$ to be the direct sum of as many copies of $\bigoplus\limits_{r \in \Reg(R)} (R/Rr)$ as the minimum cardinality of the generating system of $M_{\alpha+1}/M_{\alpha}$. If $B$ is a Baer module, then the set $E$ defined in Lemma \ref{filt} is not stationary in $\kappa$; thus, there exists a club $C$ in $\kappa$ which fails to intersect $E$. By keeping only the $M_{\alpha}$ with $\alpha \in C$ and renaming the elements of $C$ by the ordinal $<\kappa$, we obtain a chain \ref{(1)} such that $\Ext_R^1(M_{\alpha+1}/M_{\alpha}, H_{\alpha+1}/H_{\alpha}) = 0$ for all $\alpha < \kappa$. Because of the choice of $G_{\alpha}$, $M_{\alpha+1}/M_{\alpha}$ is a Baer module for every $\alpha < \kappa$ by Lemma \ref{1bapd1} and Proposition \ref{2bapd1}. Since $M_{\alpha+1}/M_{\alpha}$ is less than $\kappa$-generated, the induction applies, and we can insert between $M_{\alpha}$ and $M_{\alpha+1}$ a continuous well-ordered ascending chain of submodules with countably generated factors which are all Baer modules. When this is done for all $\alpha < \kappa$, we finally get the chain we want.

\textbf{Case 2}: $\kappa$ is a singular cardinal. Let $\mathscr{N}$ be the set of all countably generated Baer $R$-modules. Since $R$ is a $\tau_q$-semisimple ring, every countably generated Baer $R$-module is countably presented by Lemma \ref{gencard}. Since it holds for arbitrary rings, the rest is the same as in \cite[Proof of Theorem A, Case 2]{EFS90}.
\end{proof}

Now, we are ready to prove our main result.

\begin{theorem}\label{main}
Let $R$ be a $\tau_q$-semisimple ring. Then every Baer $R$-module is projective.
\end{theorem}

\begin{proof}
Note that every Baer module is a well-ordered continuous ascending chain $(B_{\alpha} \mid \alpha < \kappa)$ of submodules such that every factor $B_{\alpha+1}/B_{\alpha}$ is a countably generated Baer module by Theorem \ref{Baer-filt}. So it suffices to prove that every countably generated Baer module is projective by the Auslander-Eklof lemma.

Let $B$ be a countably generated Baer module. Then $B$ is flat and of projective dimension at most $1$ by Proposition \ref{baerflat} and Proposition \ref{baerpd1}. Note that $B$ is countably presented by Lemma \ref{gencard}. So, by \cite[Proposition 3.1]{HBH08}, there is a countable direct system of finitely generated free modules
$$F_1\xrightarrow{f_1}F_2\xrightarrow{f_2}F_3\rightarrow\cdots\rightarrow F_n\xrightarrow{f_n} F_{n+1}\rightarrow\cdots$$
such that $\lim\limits_{\longrightarrow} F_n \cong B$.

Set $M := \bigoplus\limits_{r \in \Reg(R)} R/Rr$. Then $M$ and $M^{(\mathbb{N})}$ are torsion modules, and so $\Ext_R^1(B, M^{(\mathbb{N})}) = 0$. It follows by \cite[Proposition 2.5]{HBH08} that the tower
$$\big(\Hom_R(F_n,M),\Hom_R(f_n, M)\big)_{n\in\mathbb{N}}$$
satisfies the Mittag-Leffler condition. By \cite[Corollary 2.6]{HBH08}, the tower
$$\big(\Hom_R(F_n,\prod\limits_{r\in\Reg(R)}R/Rr),\Hom_R(f_n, \prod\limits_{r\in\Reg(R)}R/Rr)\big)_{n\in\mathbb{N}}$$
also satisfies the Mittag-Leffler condition.

Note that every $\tau_q$-semisimple ring can be written as a direct sum of a semisimple ring and a $\tau_q$-semisimple ring that does not contain a field as a direct summand. Every Baer module over a semisimple ring is certainly projective, so we can assume that $R$ does not contain a field as a direct summand. Thus, $R$ is a pure submodule of $\prod\limits_{r \in \Reg(R)} R/Rr$ by Proposition \ref{fg1}. Applying \cite[Corollary 2.9]{HBH08}, we conclude that the tower
$$\big(\Hom_R(F_n,R),\Hom_R(f_n, R)\big)_{n\in\mathbb{N}}$$
satisfies the Mittag-Leffler condition. Thus, by \cite[Proposition 2.5]{HBH08}, we have that $B$ is projective.
\end{proof}

\begin{remark}
{\rm Note that not every Baer module is projective in general. Indeed, let $R$ be a total ring of quotients but not semisimple. Then every $R$-module is a Baer module. So there are Baer modules which are not projective. In general, we propose the following conjecture:

\medskip

\noindent\textbf{Conjecture:} A ring $R$ is a $\tau_q$-semisimple ring if and only if every Baer $R$-module is projective.
}
\end{remark}

\medskip

\noindent{\bf Added in the Proof}

We became aware that the concepts of $\tau_q$-semisimple rings and semiprime Goldie rings are indeed equivalent. Specifically, \cite[Theorem 14.39]{gt} states, ``If $R$ is a semiprime Goldie ring, then every Baer $R$-module is projective." Prior to this, we had assumed these notions to be distinct. Although our proof diverges fundamentally from the aforementioned proof, we acknowledge the overlap in the underlying concepts. Our approach, while distinct, does leverage the methodology utilized in the proof concerning integral domains. Specifically, our proof for $\tau_q$-semisimple rings was developed independently and utilizes a different framework, focusing on the extension of Baer modules' properties within this ring class.

\begin{acknowledgement}
The second author was supported by the Basic Science Research Program through the National Research Foundation of Korea (NRF), funded by the Ministry of Education (2021R1I1A3047469).
\end{acknowledgement}


\begin{thebibliography}{99}

\bibitem{HBH08} L. Angeleri H\"ugel, S. Bazzoni and D. Herbera, A solution to the Baer splitting problem, \textit{Trans. Amer. Math. Soc.}, \textbf{360} (2008) 2409--2421.

\bibitem{B36} R. Baer, The subgroup of the elements of finite order of an abelian group, \textit{Ann. Math.}, \textbf{37} (1936) 766--781.

\bibitem{EF88} P. C. Eklof and L. Fuchs, Baer modules over valuation domains, \textit{Ann. Mat. Pura Appl. IV. Ser.}, \textbf{150} (1988) 363--373. \href{https://doi.org/10.1007/BF01761475}{DOI:10.1007/BF01761475}.

\bibitem{EFS90} P. C. Eklof, L. Fuchs and S. Shelah, Baer modules over domains, \textit{Trans. Amer. Math. Soc.}, \textbf{322} (1990) 547--560.

\bibitem{F23} L. Fuchs, Covers and envelopes related to divisibility, \textit{Bol. Soc. Mat. Mex.}, III. Ser. \textbf{29}(3) (2023) Paper No. 75, 11 pages.

\bibitem{FS01} L. Fuchs and L. Salce, \textit{Modules over Non-Noetherian Domains}, Providence: AMS, 2001.

\bibitem{gt} R. G\"obel and J. Trlifaj, \textit{Approximations and Endomorphism Algebras of Modules}, De Gruyter Exp. Math., vol. 41, Berlin: Walter de Gruyter GmbH \& Co. KG, 2012.

\bibitem{G86} J. S. Golan, \textit{Torsion Theories}, Longman Scientific and Technical, New York, 1986.

\bibitem{K62} I. Kaplansky, The splitting of modules over integral domains, \textit{Arch. Math.}, \textbf{13} (1962) 341--343.

\bibitem{M83} E. Matlis, The minimal prime spectrum of a reduced ring, \textit{Illinois J. Math.}, \textbf{27} (1983) 353--391.

\bibitem{PR04} G. Puninski and P. Rothmaler, When every finitely generated flat module is projective, \textit{J. Algebra}, \textbf{277} (2004) 542--558.

\bibitem{WK24} F. G. Wang and H. Kim, \textit{Foundations of Commutative Rings and Their Modules}, 2nd ed., Algebra and Applications, vol. 31, Springer, Singapore, 2024. \href{https://doi.org/10.1007/978-981-97-5284-3}{DOI:10.1007/978-981-97-5284-3}.

\bibitem{fl11} F. G. Wang and J. L. Liao, $S$-injective modules and S-injective envelopes, \textit{Acta Math. Sinica (Chin. Ser.)}, \textbf{52} (2011) 271--284 (in Chinese).

\bibitem{wzcc20} F. G. Wang, D. C. Zhou and D. Chen, Module-theoretic characterizations of the ring of finite fractions of a commutative ring, \textit{J. Commut. Algebra}, \textbf{14}(1) (2022) 141--154. \href{https://doi.org/10.1216/jca.2022.14.141}{DOI:10.1216/jca.2022.14.141}.

\bibitem{ZQ23} X. L. Zhang and W. Qi, On $\tau_q$-flatness and $\tau_q$-coherence, \textit{J. Algebra Appl.}, to appear. \href{https://doi.org/10.1142/S0219498826500106}{DOI:10.1142/S0219498826500106}.

\bibitem{z-q-proj-sim} X. L. Zhang, On $\tau_q$-projectivity and $\tau_q$-simplicity, \textit{J. Algebra Appl.}, to appear.

\bibitem{ZDC20} D. C. Zhou, H. Kim, F. G. Wang and D. Chen, A new semistar operation on a commutative ring and its applications, \textit{Comm. Algebra}, \textbf{48}(9) (2020) 3973--3988. \href{https://doi.org/10.1080/00927872.2020.1753060}{DOI:10.1080/00927872.2020.1753060}.



\end{thebibliography}
\end{document}